\documentclass[12pt, reqno]{amsart}
\usepackage{ amsmath,amsthm, amscd, amsfonts, amssymb, graphicx, color}
\usepackage[bookmarksnumbered, colorlinks, plainpages]{hyperref}
\newtheorem{thm}{Theorem}[section]
\newtheorem{cor}[thm]{Corollary}
\newtheorem{lem}[thm]{Lemma}

\newtheorem{exam}[thm]{Example}
\numberwithin{equation}{section}

\begin{document}

\title{Generalized Cline's formula for g-Drazin inverse}

\author{Huanyin Chen}
\author{Marjan Sheibani}
\address{
Department of Mathematics\\ Hangzhou Normal University\\ Hang -zhou, China}
\email{<huanyinchen@aliyun.com>}
\address{Women's University of Semnan (Farzanegan), Semnan, Iran}
\email{<sheibani@fgusem.ac.ir>}

\subjclass[2010]{15A09, 47A11, 47A53, 16U99.} \keywords{Cline's formula; Drazin inverse; Generalized Drazin inverse; Matrix; Bounded operator; Common spectral properties.}

\begin{abstract}
Let $R$ be an associative ring with an identity and suppose that $a,b,c,d\in R$ satisfy $$bdb=bac, dbd=acd.$$ If $ac$ has generalized Drazin ( respectively, pseudo Drazin, Drazin) inverse, we prove that $bd$ has generalized Drazin (respectively, pseudo Drazin, Drazin) inverse. In particular, as applications, we obtain new common spectral properties of bounded linear operators over Banach spaces.\end{abstract}

\maketitle

\section{Introduction}

Let $R$ be an associative ring with an identity. The commutant of $a\in R$ is defined by $comm(a)=\{x\in
R~|~xa=ax\}$. The double commutant of $a\in R$ is defined by $comm^2(a)=\{x\in R~|~xy=yx~\mbox{for all}~y\in comm(a)\}$.

An element $a\in R$ has Drazin inverse in case there exists $b\in R$ such that $$b=bab, b\in comm^2(a), a-a^2b\in R^{nil}.$$ The preceding $b$ is unique if exists, we denote it by $a^{D}$. Let $a,b\in R$. Then $ab$ has Drazin inverse if and only if $ba$ has Drazin inverse and $(ba)^{D}=b((ab)^{D})^2a$. This was known as Cline's formula for Drazin inverse. Cline's formula plays an elementary role in matrix and operator theory (see~\cite{Ba,Be,KT,L,LC,Mi,Y,YF,Z,Z1,Z2}).

An element $a\in R$ has g-Drazin inverse (i.e., generalized Drazin inverse) in case there exists $b\in R$ such that $$b=bab, b\in comm^2(a), a-a^2b\in R^{qnil}.$$ The preceding $b$ is unique if exists, we denote it by $a^{d}$. Here, $R^{qnil}=\{a\in R~|~1+ax\in R^{inv}~\mbox{for
every}~x\in comm(a)\}$. For a Banach algebra $\mathcal{A}$ it is well known
 that $$a\in \mathcal{A}^{qnil}\Leftrightarrow
\lim\limits_{n\to\infty}\parallel a^n\parallel^{\frac{1}{n}}=0\Leftrightarrow \lambda-a\in \mathcal{A}^{inv}~\mbox{for any scarlar}~\lambda\neq 0.$$ Many papers discussed Cline's formula for g-Drazin inverse in the setting of matrices, operators, elements of Banach algebras or rings. For any $a,b\in R$, Liao et al. proved that $ab$ has g-Drazin inverse if and only if $ba$ has g-Drazin inverse and $(ba)^{d}=b((ab)^{d})^2a$. This was known as Cline's formula for g-Drazin inverse (see~\cite[Theorem 2.1]{LC}). In ~\cite[Theorem 2.3]{L}, Lian and Zeng extended Cline's formula for a generalized Drazin inverse to the case when $aba=aca$. As a further extension, Miller and Zguitti generalized Cline's formula for g-Drazin inverse under the condition
$$dba=aca, dbd=acd~(\mbox{see~\cite[Theorem 3.2]{Mi}}).$$ This was independently proved by Zeng et al. (see~\cite[Theorem 2.7]{Z2}).

For any $a,b\in R$, Jacobson's Lemma for g-Drazin inverse states that $1+ab\in R^d$ if and only if $1+ba\in R^d$ (see~\cite[Theorem 2.3]{Z3}). Corach et al. \cite{C} generalized Jacobson's Lemma to the case that $aba=aca$. Recently, Yan et al. extended Jacobson's Lemma to the case $dba=aca, dbd=acd~(\mbox{see~\cite[Theorem 3.1]{YZ}}).$ In ~\cite{M}, Mosic studied the generalization of Jacobson's Lemma for g-Drazin inverses under a new condition
$$bdb=bac, dbd=acd~(\mbox{see~\cite[Theorem 2.5]{M}}).$$
The motivation of this paper is to investigate whether Cline's formula holds under the preceding Mosic's condition.
In Section 2, we proved that if $ac$ has g-Drazin inverse then $bd$ has g-Drazin inverse under the condition $bdb=bac, dbd=acd$.

Recall that an element $a$ in a ring $R$ has pseudo Drazin (i.e., p-Drazin inverse) provided that there exists $b\in R$ such that $$b=bab, b\in comm^2(a),
a^k-a^{k+1}b \in R^{rad}$$ for some $k\in {\Bbb N}$. The preceding $b$ is unique if exists, we denote it by $a^{\dag}$. The smallest $k$ for which the preceding holds is called the p-Drazin index of $a$ and denote it by $i(a)$. The p-Drazin inverse is an intermediary between the Drazin and generalized Drazin inverses. In Section 3, as consequences of our main result, we investigate the corresponding Cline's formula for p-Drazin and Drazin inverses.

In Section 4, as applications of our main result, we determine the common spectral properties of bounded linear operators over Banach spaces. Let $A,B,C,D\in \mathcal{L}(X)$ such that $BDB=BAC,$ $DBD=ACD$. We prove that $\sigma_d(AC)=\sigma_d(BD),$ where $\sigma_d$ is the g-Drazin spectrum.

Throughout the paper, all rings are associative with an identity. We use $R^{inv}, R^{nil}$ and $R^{rad}$ to denote the set
of all units, the set of all nilpotents and the Jacobson radical of the ring $R$, respectively. $R^{D}$ and $R^{d}$ denote the sets of all Drazin and g-Drazin invertible elements in $R$. ${\Bbb C}$ stands for the field of all complex numbers.

\section{generalized Cline's Formula}

For any elements $a,b$ in a ring $R$, it is well known that $ab\in R^{qnil}$ if and only if $ba\in R^{qnil}$ (see~\cite[Lemma 2.2]{L}). We now generalized this fact as follows.

\begin{lem} Let $R$ be a ring, and let $a,b,c,d\in R$ satisfying $bdb=bac, dbd=acd$. If $ac\in R^{qnil}$, then $bd\in R^{qnil}$.\end{lem}
\begin{proof} Let $x\in comm(bd)$. Then we have
$$\begin{array}{lll}
(dx^3bac)ac&=&(dx^3bdb)ac\\
&=&(dbdx^3)(bac)\\
&=&(acdx^3)(bac)\\
&=&ac(dx^3bac).
\end{array}$$
Hence $dx^3bac\in comm(ac)$, and so $1-(dx^3bac)ac\in R^{inv}$. By using Jacobson's Lemma, $1-x^3bdbdbd=1-x^3bacacd\in R^{inv}$. Then
$$\begin{array}{lll}
(1-xbd)(1+xbd+x^2bdbd)&=&(1+xbd+x^2bdbd)(1-xbd)\\
&=&1-x^3bdbdbd\\
&\in& R^{inv},
\end{array}$$ therefore $1-xbd\in R^{inv}$. This shows that $bd\in R^{qnil}.$\end{proof}

We come now to the main result of this paper.

\begin{thm} Let $R$ be a ring, and let $a,b,c,d\in R$ satisfying $bdb=bac, dbd=acd$. If $ac\in R^{d}$, then $bd\in R^{d}$ and
$(bd)^{d}=b((ac)^{d})^2d$.\end{thm}
\begin{proof} Suppose that $ac$ has g-Drazin inverse and $(ac)^{d}=h$. Let $e=bh^2d$ and $t\in comm(bd)$. Then $$ac(dtbac)=d(bd)tbac=dt(bdb)(ac)=dt(bac)(ac)=(dtbac)ac.$$
Thus $dtbac\in comm(ac)$, and so $(dtbac)h=h(dtbac).$ Therefore we have
 $$\begin{array}{lll}
 et&=&bh^4(ac)(ac)dt\\
 &=&bh^4d(bd)^2t\\
 &=&bh^4dt(bdb)d\\
 &=&bh^3(hdtbac)d\\
 &=&b(dtbac)h^4d\\
 &=&tb(ac)^2h^4d\\
 &=&te
 \end{array}$$
 Hence $e\in comm^2(bd)$.

Since $bd\in comm(bd)$, by the preceding discussion, we have $dbdbac\in comm(ac)$, and then $$(db)^3(ac)=dbdbac(ac)=(ac)dbdbac=(db)^3(ac).$$
This implies that $(db)^3h=h(db)^3$, and so
 $$\begin{array}{lll}
 e(bd)e&=&bh^2d(bd)bh^2d=b(h^3ac)dbdbh^2d=bh^3(db)^3h^2d\\
 &=&bh^5dbdbdbd=bh^5(ac)^3d=bh^2d=e.
 \end{array}$$

 Let $p=1-(ac)h$. Then $(pa)c=ac-achac=ac-(ac)^2h\in R^{qnil}$.
 Moreover, we have
  $$\begin{array}{lll}
  bd-(bd)^2e&=&bd-bdbdbh^2d=bd-bdbdbh^2d\\
  &=&bd-bacach^2d=b(1-ach)d=b(pd).
  \end{array}$$
 One easily checks that
  $$\begin{array}{lll}
  (pd)b(pd)&=&(1-ac)db(1-ac)d\\
  &=&(1-ac)(dbd-dbacd)\\
  &=&(1-ac)(acd-dbdbd)\\
  &=&(1-ac)(acd-acacd)\\
  &=&(pa)c(pd).
  \end{array}$$ Moreover, we check that
  $$\begin{array}{ll}
  &b(pd)b\\
  =&b(1-ach)db=bdb-bachdb=bac-bh^3(db)^3ac\\
  =&bac-b(db)^3h^3ac=bac-(bdb)dbdbh^3ac=bac-b(ac)^3h^3ac\\
  =&bac-bhacac=b(1-ach)ac=b(pa)c.
  \end{array}$$
  Then by Lemma 2.1, $b(pd)\in R^{qnil}$. Hence $bd$ has g-Drazin inverse $e$. That is,
  $e=bh^2a=(bd)^{d},$ as desired.\end{proof}

In the case that $c=b$ and $d=a$, we recover the Cline's formula for g-Drazin inverse.

 \begin{cor} Let $R$ be a ring, and let $a,b\in R$. If $ab\in R^{d}$, then $ba\in R^{d}$ and $(ba)^{d}=b((ab)^{d})^2a$.\end{cor}

The following examples show that the preceding theorem is independent from~\cite[Theorem 3.2]{Mi} and ~\cite[Theorem 2.7]{Z2}.

\begin{exam}\end{exam} Let $R=M_2({\Bbb C})$. Choose $$\begin{array}{c}
a=
\left(
\begin{array}{cc}
0&1\\
0&0
\end{array}
\right), b=\left(
\begin{array}{cc}
1&0\\
0&0
\end{array}
\right),\\
c=\left(
\begin{array}{cc}
1&0\\
1&1
\end{array}
\right), d=\left(
\begin{array}{cc}
1&1\\
0&0
\end{array}
\right)\in R.
\end{array}$$ Then we check that
$$bdb=\left(
\begin{array}{cc}
1&0\\
0&0
\end{array}
\right)\neq \left(
\begin{array}{cc}
1&1\\
0&0
\end{array}
\right)
=bac, dbd=\left(
\begin{array}{cc}
1&1\\
0&0
\end{array}
\right)=acd,$$ but $$dba=\left(
\begin{array}{cc}
0&1\\
0&0
\end{array}
\right)=aca, dbd=acd.$$ In this case, $ac\in R^{D}$.

\begin{exam}\end{exam} Let $R=M_2({\Bbb C})$. Choose $$\begin{array}{c}
a=
\left(
\begin{array}{cc}
0&1\\
0&0
\end{array}
\right), b=\left(
\begin{array}{cc}
0&0\\
0&1
\end{array}
\right),\\
c=\left(
\begin{array}{cc}
1&0\\
1&1
\end{array}
\right), d=\left(
\begin{array}{cc}
1&0\\
-1&0
\end{array}
\right)\in R.
\end{array}$$ Then we check that
$$bdb=0=bac, dbd=0=acd,$$ but $$dba=0\neq \left(
\begin{array}{cc}
0&1\\
0&0
\end{array}
\right)=aca, dbd=acd.$$ In this case, $ac\in R^{D}$.

\section{p-Drazin inverses}

In this section, we investigate the Cline's formula for the p-Drazin inverse. The following lemma is crucial.

\begin{lem} Let $R$ be a ring, and let $a\in R$. If $a\in R^{\dag}$, then $a\in R^{d}$ and $a^{\dag}=a^d$.\end{lem}
\begin{proof} By hypothesis, there exists $b\in comm^2(a)$ such that $b=b^2a, a^k-a^{k+1}b
 \in R^{rad}$ for some $k\in {\Bbb N}$. Hence, $(a-a^2b)^k=a^k(1-ab)^k=a^k(1-ab)=a^k-a^{k+1}b\in R^{rad}$. Thus, $a-a^2b\in R^{qnil}$. Therefore $a$ has
 g-Drazin inverse $b$. In light of~\cite[Lemma 2.4]{K}, $b$ is unique, and so $b=a^{\dag}$, as required.\end{proof}

\begin{thm} Let $R$ be a ring, and let $a,b,c,d\in R$ satisfying $bdb=bac, dbd=acd$. If $ac\in R^{\dag}$,
$bd\in R^{\dag}$, then $(bd)^{\dag}=b((ac)^{\dag})^2d$ and $i(bd)\leq i(ac)+1$.
\end{thm}
\begin{proof} Suppose that $ac\in R^{\dag}$. By virtue of Lemma 3.1, $(ac)^d=(ac)^{\dag}$. In view of Theorem 2.2, $bd\in R^{d}$ and $(bd)^{d}=b((ac)^{\dag})^2d$. One easily checks that $$(bd)^{k+2}(bd)^{\dag}-(bd)^{k+1}=b\big((ac)^{k+1}(ac)^d-(ac)^k\big))$$ for all $k\in {\Bbb N}$. Therefore $(bd)^m-(bd)^{m+1}(bd)^{\dag}\in R^{rad}$ for some $m\in {\Bbb N}$, and so $bd\in R^{\dag}$. Moreover, $(bd)^{\dag}=(bd)^{d}=b((ac)^{\dag})^2d$, as desired.\end{proof}

\begin{cor} Let $R$ be a ring, and let $a,b\in R$. If $ab\in R^{\dag}$, then $ba\in R^{\dag}$ and $(ba)^{\dag}=b((ab)^{\dag})^2a$.\end{cor}
\begin{proof} This is obvious by choosing $c=b$ and $d=a$ in Theorem 3.2.\end{proof}

\begin{thm} Let $R$ be a ring, and let $a,b,c,d\in R$ satisfying $bdb=bac, dbd=acd$. If $ac\in R^{D}$, then $(bd)^{D}=b((ac)^{D})^2d$ and $i(bd)\leq i(ac)+1$.\end{thm}
\begin{proof} Suppose that $ac\in R^{D}$. Then $ac\in R^{dag}$. In view of Theorem 3.2, $bd\in R^{\dag}$, $(bd)^{\dag}=b((ac)^{D})^2d$ and $i(bd)\leq i(ac)+1$. As in the proof of Theorem 3.2, $(bd)^{k+2}(bd)^{\dag}-(bd)^{k+1}=b\big((ac)^{k+1}(ac)^D-(ac)^k\big))$ for all $k\in {\Bbb N}$. Therefore $(bd)^m=(bd)^{m+1}(bd)^{\dag}$ for some $m\in {\Bbb N}$, and so $bd\in R^{D}$. Furthermore, $(bd)^{D}=(bd)^{\dag}=b((ac)^{D})^2d$, as required.\end{proof}

Recall that $a\in R$ has group inverse if $a$ has Drazin inverse with index $1$, and we denote the group inverse of $a$ by $a^{\#}$. We now derive

\begin{cor} Let $R$ be a ring, and let $a,b,c,d\in R$ satisfying $bdb=bac, dbd=acd$. If $ac$ has group inverse, then \end{cor}
\begin{enumerate}
\item [(1)]{\it $bd\in R^{inv}$; or}
\vspace{-.5mm}
\item [(2)]{\it $bd$ has group inverse and $(bd)^{\#}=b((ac)^{\#})^2d$; or}
\vspace{-.5mm}
\item [(3)]{\it $bd\in R^{d}$ and $(bd)^{d}=b((ac)^{\#})^2d$.}
\end{enumerate}
\begin{proof} Since $ac$ has group inverse, it follows by Theorem 3.4 that $i(bd)\leq 2$. Hence, $i(bd)=0, 1, 2$. This completes the proof.\end{proof}

\begin{exam}\end{exam} Let $R=M_2({\Bbb Z})$ be the ring of all $2\times 2$ integer matrices. Choose $$a=
\left(
\begin{array}{cc}
0&1\\
0&1
\end{array}
\right), b=
\left(
\begin{array}{cc}
1&1\\
0&0
\end{array}
\right), c=
\left(
\begin{array}{cc}
1&-1\\
0&0
\end{array}
\right), d=a.$$ Then $bdb=bac, dbd=acd$. In this case, $ac=0$ has group inverse, but $bd=\left(
\begin{array}{cc}
0&2\\
0&0
\end{array}
\right)$ has no group inverse in $R$.

\section{Applications}

Let $X$ be  Banach space, and let $\mathcal{L}(X)$ denote the set of all bounded linear operators from Banach space to itself, and let $a\in \mathcal{L}(X)$. The Drazin spectrum $\sigma_D(a)$ and g-Drazin spectrum $\sigma_{d}(a)$ are defined by $$\begin{array}{c}
\sigma_D(a)=\{ \lambda\in {\Bbb C}~|~\lambda-a\not\in A^{D}\};\\
\sigma_{d}(a)=\{ \lambda\in {\Bbb C}~|~\lambda-a\not\in A^{d}\}.
\end{array}$$ The goal of this section is concern on common spectrum properties of $\mathcal{L}(X)$. We now record the generalized Jacobson's Lemma as follows.

\begin{lem}~(\cite{M}) Let $R$ be a ring, and let $a,b,c,d\in R$ satisfying $bdb=bac, dbd=acd$. If $1-ac\in R^{inv}$, then $1-bd\in R^{inv}$ and $$(1-bd)^{-1}=1+b(1-ac)^{-1}d.$$
\end{lem}

We now ready to prove the following.

\begin{thm} Let $A,B,C,D\in \mathcal{L}(X)$ such that $BDB=BAC,$ $DBD=ACD$, then $$\sigma_d(AC)=\sigma_d(BD).$$\end{thm}
\begin{proof} Case 1. $0\in \sigma_d(AC)$. Then $AC\not\in A^{d}$. In view of Theorem 2.2, $BD\in A^{d}$. Thus $0\in \sigma_d(BD)$.

Case 2. $0\not\in \lambda\in\sigma_d(AC)$. Then $\lambda\in acc\sigma(AC)$. Thus, we see that
$$\lambda=\lim\limits_{n\to \infty}\{ \lambda_n ~|~ \lambda_n I-AC\not\in \mathcal{L}(X)^{-1}\}.$$
For $\lambda_n\neq 0$, it follows by Lemma 4.1 that $I-(\frac{1}{\lambda_n} A)C\in \mathcal{L}(X)^{-1}$ if and only if $I-B(\frac{1}{\lambda_n}D)\in \mathcal{L}(X)^{-1}$. Therefore
$$\lambda=\lim\limits_{n\to \infty}\{ \lambda_n ~|~ \lambda_n I-BD\not\in \mathcal{L}(X)^{-1}\}\in acc\sigma(BD)=\sigma_d(BD).$$
Therefore $\sigma_d(AC)\subseteq \sigma_d(BD).$ Likewise, $\sigma_d(BD)\subseteq \sigma_d(AC)$, as required.\end{proof}

\begin{cor} Let $A,B\in \mathcal{L}(X)$ , then $$\sigma_d(AB)=\sigma_d(BA).$$\end{cor}
\begin{proof} This is obvious by choosing $C=B$ and $D=A$ in Theorem 4.2.\end{proof}

For the Drazin spectrum $\sigma_D(a)$, we now derive

\begin{thm} Let $A,B,C,D\in \mathcal{L}(X)$ such that $BDB=BAC,$ $DBD=ACD$, then $$\sigma_D(AC)=\sigma_D(BD).$$\end{thm}
\begin{proof} In view of Theorem 3.4, $AC\in \mathcal{L}(X)^{D}$ if and only if $BD\in \mathcal{L}(X)^{D}$, and therefore we complete the proof
by~\cite[Theorem 3.1]{Y}.\end{proof}

A bounded linear operator $T\in \mathcal{L}(X)$ is Fredholm operator if $dimN(T)$ and $codimR(T)$ are finite, where $N(T)$ and $R(T)$ are the null space and the range of $T$ respectively. If furthermore the Fredholm index $ind(T)=0$, then $T$ is said to be Weyl operator. For each nonnegative integer $n$ define $T_{|n|}$ to be the restriction of $T$ to $R(T^n)$. If for some $n$, $R(T^n)$ is closed and $T_{|n|}$ is a Fredholm operator then $T$ is called a $B$-Fredholm operator. $T$ is said to be a $B$-Weyl operator if $T_{|n|}$ is a Fredholm operator of index zero (see~\cite{Ba}). The $B$-Fredholm and $B$-Weyl spectrums of $T$ are defined by
$$\begin{array}{c}
\sigma_{BF}(T)=\{ \lambda\in {\Bbb C}~|~T-\lambda I~\mbox{is not a}$ $B-\mbox{Fredholm operator}\};\\
\sigma_{BW}(T)=\{ \lambda\in {\Bbb C}~|~T-\lambda I~\mbox{is not a}$ $B-\mbox{Weyl operator}\}.
\end{array}$$

\begin{cor} Let $A,B,C,D\in \mathcal{L}(X)$ such that $BDB=BAC,$ $DBD=ACD$, then $$\sigma_{BF}(AC)=\sigma_{BF}(BD).$$\end{cor}
\begin{proof} Let $\pi: \mathcal{L}(X)\to \mathcal{L}(X)/F(X)$ be the canonical map and $F(X)$ be the ideal of finite rank operators in $\mathcal{L}(X)$. As in well known, $T\in \mathcal{L}(X)$ is $B$-Fredholm if and only if $\pi(T)$ has Drazin inverse. By hypothesis, we see that
$$\begin{array}{c}
\pi(B)\pi(D)\pi(B)=\pi(B)\pi(A)\pi(C),\\
\pi(D)\pi(B)\pi(D)=\pi(A)\pi(C)\pi(D).
\end{array}$$
According to Theorem 3.4, for every scalar $\lambda$, we have
$$\lambda I-\pi(AC)~\mbox{has Drazin inverse} ~\Longrightarrow \lambda I-\pi(BD)~\mbox{has Drazin inverse}.$$ This completes the proof.\end{proof}

\begin{cor} Let $A,B,C,D\in \mathcal{L}(X)$ such that $BDB=BAC,$ $DBD=ACD$, then $$\sigma_{BW}(AC)=\sigma_{BW}(BD).$$\end{cor}
\begin{proof} If $T$ is $B$-Fredholm then for $\lambda\neq 0$ small enough, $T-\lambda I$ is Fredholm and $ind(T)=ind(T-\lambda I)$.
As in the proof of \cite[Lemma 2.3, Lemma 2.4]{Y}, we see that $I-AC$ is Fredholm if and only if $I-BD$ is Fredholm and in this case, $ind(I-AC)=ind(I-BD)$. Therefore we complete the proof by Theorem 4.4.\end{proof}

An element $a\in \mathcal{L}(X)$ is algebraic if there exists a non-zero complex polynomial $p$ such that $p(a)=0$. As an immediate consequence of Theorem 4.4, we have

\begin{cor} Let $A,B,C,D\in \mathcal{L}(X)$ such that $BDB=BAC,$ $DBD=ACD$, then $AC$ is algebraic if and only if $BD$ is algebraic.\end{cor}
\begin{proof} It follows immediately from ~\cite[Theorem 2.1]{Bo} and Theorem 4.4.\end{proof}

\end{document}